\documentclass[10pt]{amsart}

\usepackage{graphicx}
\usepackage{subfigure}
\usepackage{amsfonts}
\usepackage{amssymb}
\usepackage{amsmath,latexsym,amsthm}

\newtheorem{theorem}{Theorem}[section]
\newtheorem{lemma}[theorem]{Lemma}
\newtheorem{proposition}[theorem]{Proposition}
\newtheorem{corollary}[theorem]{Corollary}
\newtheorem{conjecture}[theorem]{Conjecture}

\theoremstyle{definition}
\newtheorem{definition}[theorem]{Definition}

\theoremstyle{remark}

\begin{document}
\title{The $2$-adic valuation of Stirling numbers}

\author{Tewodros Amdeberhan}
\address{Department of Mathematics, Tulane  University, New Orleans, LA 70118}
\email{tamdeberha@math.tulane.edu}

\author{Dante V. Manna}
\address{Department of Mathematics and Statistics, Dalhousie University, Halifax, Nova Scotia, Canada, B3H 3J5}
\email{dantemanna@gmail.com}

\author{Victor H. Moll}
\address{Department of Mathematics, Tulane  University, New Orleans, LA 70118}
\email{vhm@math.tulane.edu}

\medskip

\begin{abstract}
We analyze properties of the $2$-adic valuations of the Stirling numbers 
of the second kind.
\end{abstract}

\maketitle

\newcommand{\realpart}{\mathop{\rm Re}\nolimits}
\newcommand{\imagpart}{\mathop{\rm Im}\nolimits}

\numberwithin{equation}{section}

\section{Introduction} \label{intro} 
\setcounter{equation}{0}

Divisibility properties of integer sequences have always  been objects of
interest for number theorists. Nowadays these are expressed in terms 
$p$-adic valuations. Given a prime $p$ 
and a positive integer $m$, there exist unique integers $a, \, n$, 
with $a$ not divisible by $p$ and $n \geq 0$, such that $m = ap^{n}$. The 
number $n$ is called the {\em $p$-adic valuation} of $m$, denoted by
$n = \nu_{p}(m)$.  Thus, 
$\nu_{p}(m)$ is the highest power of $p$ that divides $m$. 
The graph in figure \ref{val-1} shows the function $\nu_{2}(m)$. 

{{
\begin{figure}[ht]
\begin{center}
\includegraphics[width=3in]{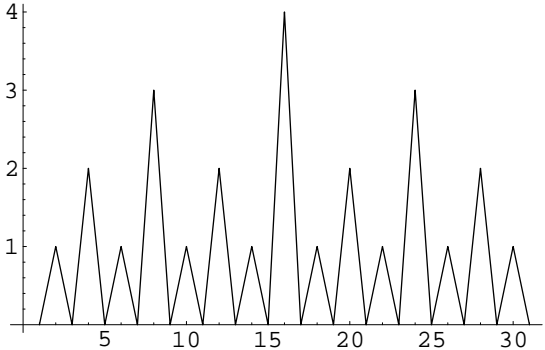}
\caption{The $2$-adic valuation of $m$}
\label{val-1}
\end{center}
\end{figure}
}}

One of the most celebrated examples is the expression for the $p$-adic 
valuation of factorials. This is due to Legendre \cite{legendre1}, who 
established
\begin{equation}
\nu_{p}(m!) = \frac{m- s_{p}(m)}{p-1}.
\end{equation}
Here $s_{p}(m)$ is the sum of the base $p$-digits of $m$. In particular,
\begin{equation}
\nu_{2}(m!) = m - s_{2}(m).  \label{legend-1}
\end{equation}
\noindent
The reader will find in \cite{graham1} details about this identity. Figure 
\ref{val-2} shows the graph of $\nu_{2}(m!)$ exhibiting its 
linear growth: $\nu_{2}(m!) \sim m$. This can be verified using the 
binary expansion of $m$:
\begin{equation}
m = a_{0} + a_{1} \cdot 2 + a_{2} \cdot 2^{2} + \ldots + 
a_{r} \cdot 2^{r}, \quad \text{with } \quad a_{j} \in \{ 0, \, 1 \}, \, 
a_{r} \neq 0,
\nonumber 
\end{equation}
\noindent
so that $2^{r} \leq m < 2^{r+1}$. Therefore $s_{2}(m) = O( \log_{2}(m))$ and 
we have
\begin{equation}
\lim\limits_{m \to \infty} \frac{\nu_{2}(m!)}{m} = 1. 
\end{equation}
The error term $s_{2}(m) = m - \nu_{2}(m!)$, shown in Figure 
\ref{val-2}, shows a regular pattern.

{{
\begin{figure}[ht]
\subfigure{\includegraphics[width=3in]{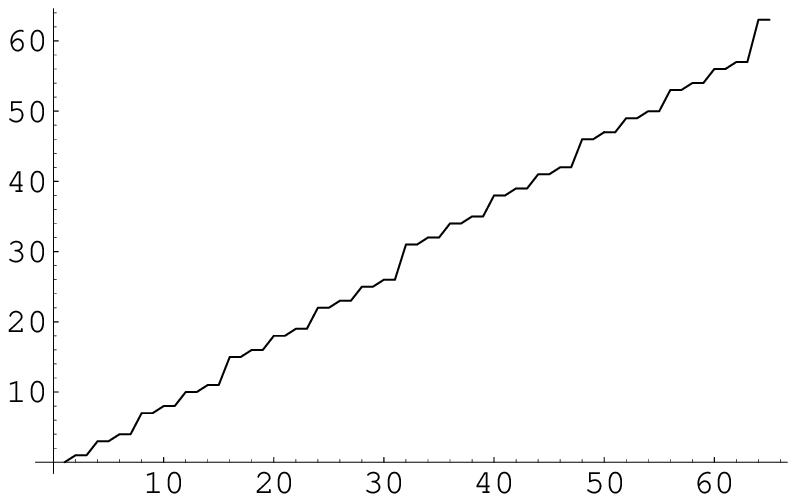}}
\subfigure{\includegraphics[width=3in]{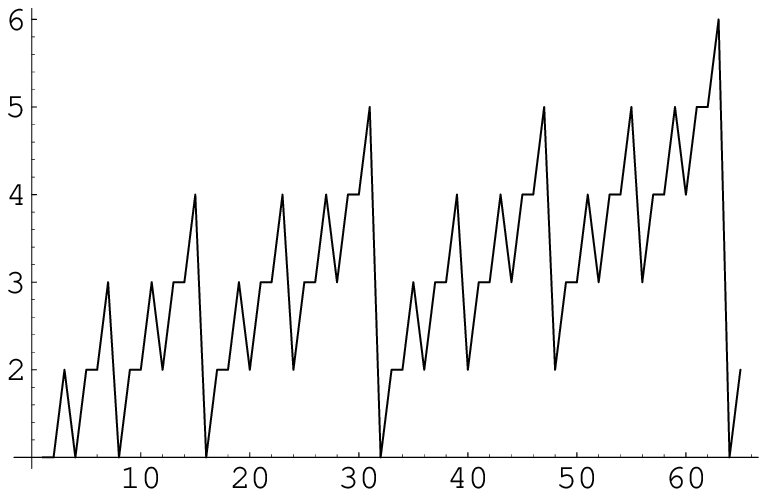}}
\caption{The $2$-adic valuation of $m!$.}
\label{val-2}
\end{figure}
}}

Legendre's result (\ref{legend-1}) provides an elementary proof of 
Kummer's identity
\begin{equation}
\nu_{2} \left( \binom{m}{k} \right) = s_{2}(k) + s_{2}(m-k) - s_{2}(m). 
\end{equation}
\noindent
Not many explicit identities of this type  are known. 

The function $\nu_{p}$ is extended to $\mathbb{Q}$ by defining
$\nu_{p} \left( \frac{a}{b} \right) = \nu_{p}(a) - \nu_{p}(b)$. The 
{\em $p$-adic metric} is then defined by
\begin{equation}
\left| r \right|_{p} := p^{-\nu_{p}(r)}, \quad \text{ for } r \in \mathbb{Q}. 
\end{equation}
It satisfies the ultrametric inequality
\begin{equation}
\left| r_{1} + r_{1} \right|_{p} \leq 
\text{Max} \left\{ \left| r_{1} \right|_{p}, \,
\left| r_{2} \right|_{p} \right\}. \label{metric1}
\end{equation}
\noindent
The completion of $\mathbb{Q}$ under this metric, denoted 
by ${\mathbb{Q}}_{p}$, is 
the field of {$p$-adic numbers}. The set ${\mathbb{Z}}_{p} := \{ x \in 
{\mathbb{Q}}_{p}: |x|_{p} \leq 1 \}$ is the ring of $p$-adic integers.  It 
plays the role of $\mathbb{Z}$ inside ${\mathbb{Q}}_{p}$. 

Our interest in $2$-adic valuations started with the sequence
\begin{equation}
b_{l,m} := \sum_{k=l}^{m} 2^{k} \binom{2m-2k}{m-k} \binom{m+k}{m} \binom{k}{l},
\label{blm-def}
\end{equation}
\noindent
for $m \in \mathbb{N}$ and $0 \leq l \leq m$, that appeared  in the 
evaluation of the definite integral 
\begin{equation}
N_{0,4}(a;m) = \int_{0}^{\infty} \frac{dx}{(x^{4} + 2ax^{2} + 1)^{m+1}}. 
\end{equation}
\noindent
In \cite{bomohyper}, it was shown that the polynomial defined by 
\begin{equation}
P_{m}(a) := 2^{-2m} \sum_{l=0}^{m} b_{l,m}a^{l}
\end{equation}
\noindent
satisfies 
\begin{equation}
P_{m}(a) = 2^{m+3/2} \, (a+1)^{m+1/2} N_{0,4}(a;m)/\pi. 
\end{equation}
\noindent
The reader will find in \cite{irrbook} more details on this integral. 

The results on the $2$-adic valuations of $b_{l,m}$ are expressed in terms of 
\begin{equation}
A_{l,m} := \frac{l! \, m!}{2^{m-l}} b_{l,m}. 
\label{A-def}
\end{equation}
\noindent
The coefficients $A_{l,m}$ can be written as 
\begin{equation}
A_{l,m} = \alpha_{l}(m) \prod_{k=1}^{m} (4k-1) - \beta_{l}(m) \prod_{k=1}^{m} 
(4k+1), 
\end{equation}
\noindent
for some polynomials $\alpha_{l}, \, \beta_{l}$ with integer coefficients and 
of degree $l$ and $l-1$, respectively.  The next remarkable property was 
conjectured in \cite{bomosha} and established by J. Little in \cite{little}. 

\begin{theorem}
All the zeros of $\alpha_{l}(m)$ and $\beta_{l}(m)$ lie on the vertical 
line $\realpart{m} = - \tfrac{1}{2}$.  
\end{theorem}

The next theorem, presented in \cite{amm1}, gives $2$-adic properties 
of $A_{l,m}$. 

\begin{theorem}
The $2$-adic valuation of $A_{l,m}$ satisfies 
\begin{equation}
\nu_{2}(A_{l,m}) = \nu_{2}( \, (m+1-l)_{2l} \, ) + l, 
\end{equation}
\noindent
where $(a)_{k} = a(a+1)(a+2) \cdots (a+k-1)$ is the Pochhammer symbol.  
\end{theorem}

The identity 
\begin{equation}
(a)_{k} = \frac{(a+k-1)!}{(a-1)!}
\end{equation}
\noindent
and Legendre's identity (\ref{legend-1}) yields the next expression for 
$\nu_{2}(A_{l,m})$. 

\begin{corollary}
The $2$-adic valuation of $A_{l,m}$ is given by 
\begin{equation}
\nu_{2}(A_{l,m}) = 3l - s_{2}(m+l) + s_{2}(m-l). 
\end{equation}
\end{corollary}

Among the other examples of $2$-adic valuations, we mention the results of
H. Cohen \cite{cohen1} on the partial sums of the 
polylogarithmic series
\begin{equation}
\text{Li}_{k}(x) := \sum_{j=1}^{\infty} \frac{x^{j}}{j^{k}}. 
\end{equation}
\noindent
Cohen proves that the sum\footnote{Cohen uses the notation 
$s_{k}(n)$, employed here in a different context.}
\begin{equation}
L_{k}(n) := \sum_{j=1}^{n} \frac{2^{j}}{j^{k}}
\end{equation}
satisfies
\begin{equation}
\nu_{2}(L_{1}( 2^{m}))  = 2^{m}+2m-4, \, \text{ for } m \geq 4,
\end{equation}
\noindent
and 
\begin{equation}
\nu_{2}(L_{2}( 2^{m}))  = 2^{m}+m-1, \, \text{ for } m \geq 4. 
\end{equation}
\noindent
The graph in figure \ref{valcohen1} shows the linear growth of 
$\nu_{2}(L_{1}(m))$ 
and the error term $\nu_{2}(L_{1}(m)) - m$.

{{
\begin{figure}[ht]
\subfigure{\includegraphics[width=3in]{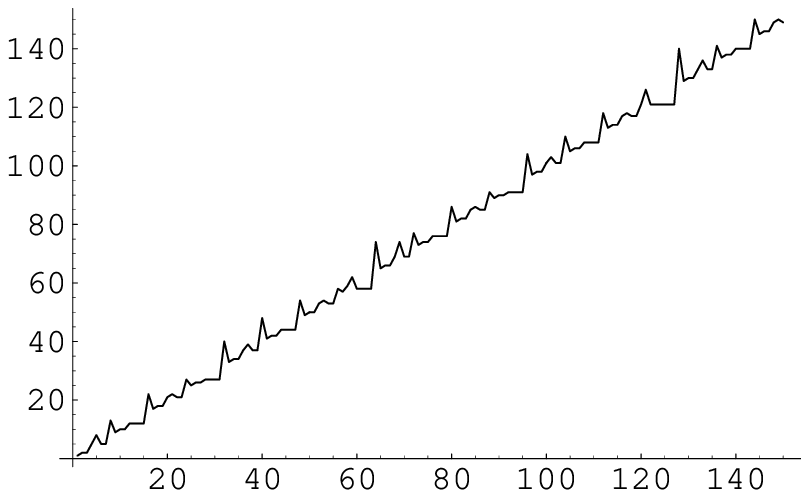}}
\subfigure{\includegraphics[width=3in]{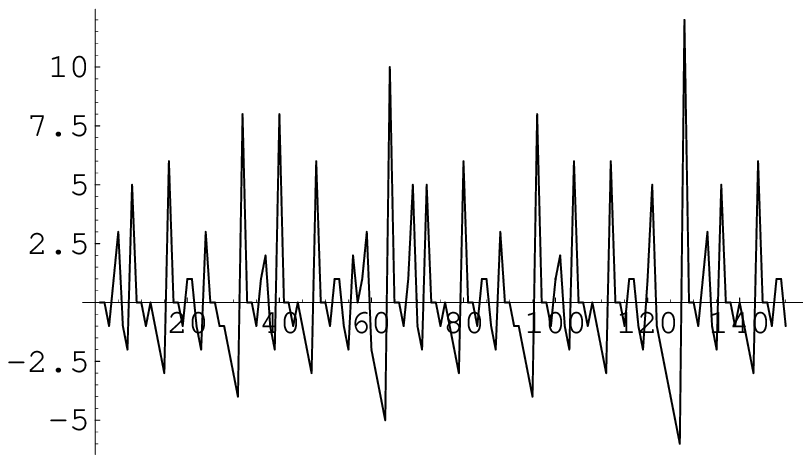}}
\caption{The $2$-adic valuation of $L_{1}(m)$}
\label{valcohen1}
\end{figure}
}}

In this paper we analyze the $2$-adic valuation of the Stirling 
numbers of the second kind $S(n,k)$, defined for $n \in 
\mathbb{N}$ and $0 \leq k \leq n$ as the number of ways to partition a set
of $n$ elements into exactly $k$ nonempty subsets. Figure \ref{fig5} shows 
the function $\nu_{2}(S(n,k))$ for $k=75$ and $k=126$. These graphs indicate 
the complexity of this problem. Section \ref{sec-pictures} gives a 
larger selection of these type of  pictures.
{{
\begin{figure}[ht]
\subfigure{\includegraphics[width=3in]{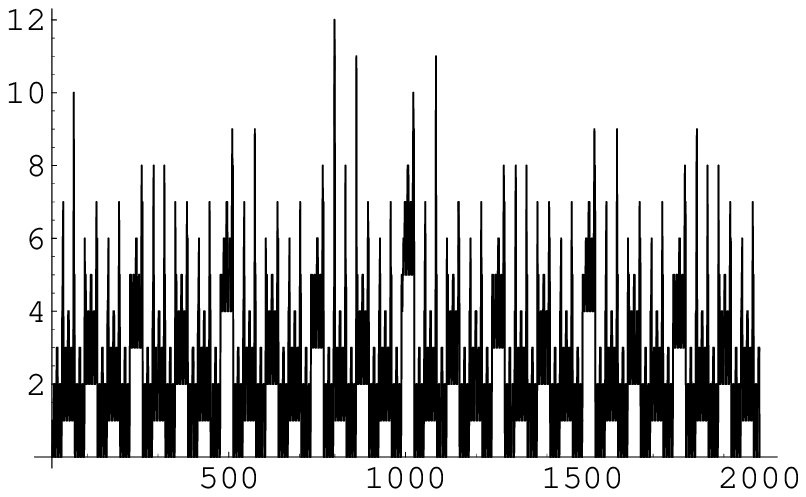}}
\subfigure{\includegraphics[width=3in]{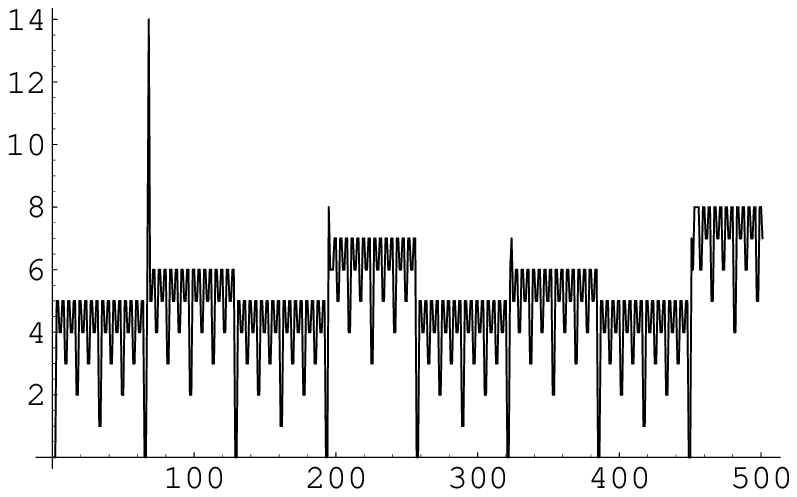}}
\caption{The data for $k=75$ and $k=126$.}
\label{fig5}
\end{figure}
}}

\vskip 0.2in 

\noindent
{\bf Main conjecture}. We describe an 
algorithm that leads to a description of the
function $\nu_{2}(S(n,k))$. Figure \ref{fig5} shows the graphs of this 
function for $k=75$ and $k=126$. The 
conjecture is stated here and the special case $k=5$ is established in 
section \ref{sec-conjecture}. 

\vskip 0.1in 

\begin{definition}
\label{def-class}
Let $k \in \mathbb{N}$ be fixed and $m \in \mathbb{N}$. Then for 
$0 \leq j < 2^{m}$ define 
\begin{equation}
C_{m,j} := \{ 2^{m}i +j: \quad \text{ for }i \in \mathbb{N} \quad 
\text{ such that } 2^{m}i+j \geq k \}. 
\end{equation}
\noindent
For example, for $k=5, m = 6$ and $j=28$ we have
\begin{equation}
C_{6,28} = \{ 2^{6}i+28: \quad i \geq 0 \}. 
\end{equation}
\noindent
We use the notation
\begin{equation}
\nu_{2}(C_{m,j}) = \{ \nu_{2} \left( S( 2^{m}i+j,k) \right): \, i \in 
\mathbb{N} \text{ and } 2^{m}i+j \geq k \}. 
\end{equation}
\end{definition}

The classes $C_{m,j}$ form a partition of $\{ n \in \mathbb{N}: \, n \geq k \}$
into classes modulo
$2^{m}$. For example, for $m=2$ and $k=5$, we have the four classes 
\begin{eqnarray}
C_{2,0} = \{ 2^{2}i: \, i \in \mathbb{N}, \, i \geq 2 \}, & \,  & 
C_{2,1} = \{ 2^{2}i+1: \, i \in \mathbb{N} \}, \nonumber \\ 
C_{2,2} = \{ 2^{2}i+2: \, i \in \mathbb{N} \},  & \, & 
C_{2,3} = \{ 2^{2}i+3: \, i \in \mathbb{N} \}. \nonumber 
\end{eqnarray}
\noindent
The class $C_{m,j}$ is called {\em constant} if $\nu_{2}(C_{m,j})$ consists 
of a single value. This single value is called the constant of the class 
$C_{m,j}$.

For example, Corollary \ref{4-vanish} shows that $\nu_{2}(S(4i+1,5)) = 0$, 
independently of $i$. Therefore, the 
class $C_{2,1}$ is constant. Similarly $C_{2,2}$ has constant valuation $0$.

We now introduce inductively the concept of {\em $m$-level}. For $m=1$, the 
$1$-level consists of the two classes 
\begin{equation}
C_{1,0} = \{ 2i: \, i \in \mathbb{N}, \, 2i \geq k \} \text{ and }
C_{1,1} = \{ 2i + 1: \, i \in \mathbb{N}, \, 2i+1 \geq k \}, 
\end{equation}
\noindent 
that is, the even and odd integers greater or equal than $k$. Assume
that the $(m-1)-$level has been defined and it consists of the $s$
classes 
\begin{equation}
C_{m-1,i_{1}}, \, C_{m-1,i_{2}}, \, \cdots, C_{m-1,i_{s}}. 
\end{equation}
\noindent
Each class $C_{m-1,i_{j}}$ splits into two classes modulo $2^{m}$, namely,
$C_{m,i_{j}}$ and $C_{m,i_{j}+2^{m-1}}$. The $m$-level is formed by
the non-constant classes modulo $2^{m}$. 

\vskip 0.1in

\noindent
{\bf Example}. We describe the case of Stirling numbers $S(n,10)$. Start 
with the fact that the $4$-level consists of the
classes $C_{4,7}, \, C_{4,8}, \, C_{4,9}$ and $C_{4,14}$. These
split into the eight classes
\begin{equation}
C_{5,7}, \, C_{5,23}, \, C_{5,8}, \, C_{5,24}, \, C_{5,9}, \, C_{5,25}, \, 
C_{5,14}, \text{ and } C_{5,30},
\nonumber
\end{equation}
\noindent
modulo $32$. Then one checks that $C_{5,23}, \, C_{5,24}, \, C_{5,25}$ 
and $C_{5,30}$ are all 
constant (with constant value $2$ for each of them). The other four 
classes form the $5$-level: 
\begin{equation}
\{ C_{5,7}, \, C_{5,8}, \, C_{5,9}, \, C_{5,14} \}. 
\end{equation}

We are now ready to state our main conjecture.

\begin{conjecture}
\label{main-conj}
Let $k \in \mathbb{N}$ be fixed. Define $m_{0} = m_{0}(k) \in \mathbb{N}$ by 
$2^{m_{0}-1} < k \leq 2^{m_{0}}$. Then the 
$2$-adic valuation of the Stirling numbers of the second kind $S(n,k)$ 
satisfies: 

\noindent
1) The first index for which there is a constant class is $m_{0}-1$. That 
is, every
class is part of the $j$-level for $1 \leq j \leq m_{0}-2$.  

\noindent
2) For any $m \geq m_{0}$, the $m$-level consists of $2^{m_{0}-2}$ classes. 
Each one of them produces a single non-constant class for the $(m+1)$-level, 
keeping the number of classes at each level constant.  
\end{conjecture}

\vskip 0.1in

\noindent
{\bf Example}. We illustrate this conjecture 
for $S(n,11)$. Here $m_{0}=4$ in view of
$2^{3} < 11 \leq 2^{4}$. The conjecture predicts that the $3$-level is 
the first with constant classes and that each level after that has exactly 
four classes. First of all, the four classes $C_{2,0}, \, 
C_{2,1}, \, C_{2,2}, \, C_{2,3}$ have non-constant $2$-adic valuation. Thus, 
they form the $2$-level. To compute the $3$-level, we observe that
\begin{equation}
\nu_{2}(C_{3,3}) = \nu_{2}(C_{3,5}) = \{ 0  \} \text{ and } 
\nu_{2}(C_{3,4}) = \nu_{2}(C_{3,6})  = \{ 1 \}, \nonumber
\end{equation}
\noindent
so there are four constant classes. The remaining four classes 
$C_{3,0}, \, C_{3,1}, \, C_{3,2}$ and $C_{3,7}$ form the $3$-level. 
Observe that each of the four classes from the $2$-level splits into a 
constant class and a class that forms part of the $3$-level. 

This process continues. At the next step, the classes of the $3$-level 
split in 
two giving a total of $8$ classes modulo $2^{4}$. For example, $C_{3,2}$ 
splits into $C_{4,2}$ and $C_{4,10}$. The conjecture states that {\em 
exactly} one of these classes has constant $2$-adic valuation. Indeed, the 
class $C_{4,2}$ satisfies $\nu_{2}(C_{4,2}) \equiv 2$. 

Figure \ref{split-11} illustrates this process.

{{
\begin{figure}[ht]
\begin{center}
\includegraphics[width=3.5in]{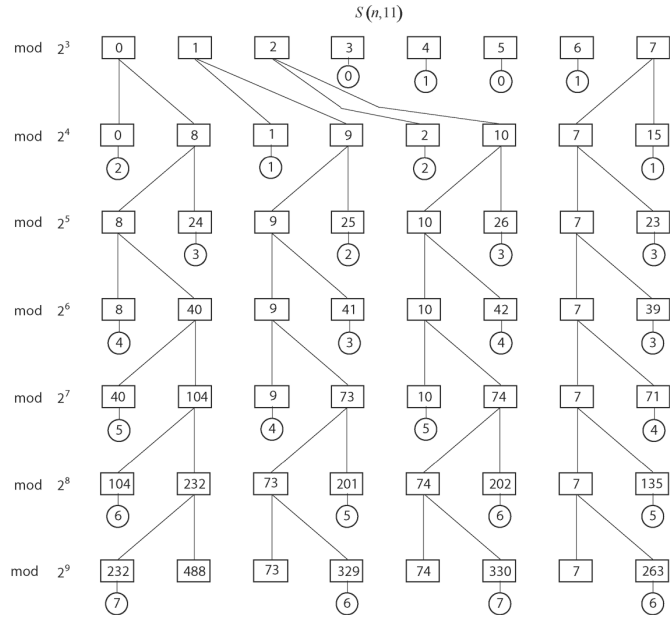}
\end{center}
\caption{The splitting for $k=11$}
\label{split-11}
\end{figure}
}}

\bigskip

\noindent
{\bf Elementary formulas}. Throughout the paper 
we will use several elementary properties of $S(n,k),$ listed below: 

\begin{itemize}
\item{ Relation to Pochhammer
\begin{equation}
x^{n} = \sum_{k=0}^{n} S(n,k) (x-k+1)_{k}
\end{equation}
}

\item{An explicit formula
\begin{equation}
S(n,k) = \frac{1}{k!} \sum_{i=0}^{k-1} (-1)^{i} \binom{k}{i} (k-i)^{n}
\label{formulastir}
\end{equation}
}

\item{The generating function
\begin{equation}
\frac{1}{(1-x)(1-2x)(1-3x) \cdots (1-kx) } = \sum_{n=1}^{\infty} S(n,k)x^{n}
\end{equation}
}

\item{The recurrence
\begin{equation}
S(n,k) = S(n-1,k-1) + kS(n-1,k)
\label{recurrence}
\end{equation}
}

\end{itemize}

Lengyel \cite{lengyel1} conjectured, and De Wannemacker \cite{wannemacker1} 
proved, a special case of the $2$-adic valuation of $S(n,k)$:
\begin{equation}
\nu_{2} \left( S(2^{n},k) \right) = s_{2}(k) -1, 
\label{leng1}
\end{equation}
\noindent
independently of $n$. Here $s_{2}(k)$ is the sum of the binary digits of $k$.
Legendre's result (\ref{legend-1}) then shows that
\begin{equation}
\nu_{2}(S(2^{n},k)) = \nu_{2}( 2^{k-1}/k!). 
\label{wanna11}
\end{equation}
\noindent
From here we obtain
\begin{equation}
\nu_{2} \left( S(2^{n}+1,k+1 \right) = s_{2}(k) -1
\label{leng2}
\end{equation}
\noindent
as a companion of (\ref{leng1}). Indeed, the recurrence 
(\ref{recurrence}) yields
\begin{equation}
S(2^{n}+1,k+1) = S(2^{n},k) + (k+1)S(2^{n},k+1), 
\end{equation}
and using (\ref{wanna11}) we have
\begin{equation}
\nu_{2}((k+1)S(2^{n},k+1))  =   \nu_{2}(2^{k}/k!) 
>  \nu_{2}( 2^{k-1}/k!) = \nu_{2}(S(2^{n},k)), \nonumber
\end{equation}
\noindent
as claimed. Similar arguments can be used to obtain the 
value of $\nu_{2}(S(m,k))$ for values of $m$ near a power of $2$. For 
instance, $\nu_{2}(S(2^{n}+2,k+2)) = s_{2}(k)-1$ if $\nu_{2}(k) \neq 0$ and 
$\nu_{2}(S(2^{n}+2,k+2)) = s_{2}(k+1)-1$ if $\nu_{2}(k+1) > 1$. 

In the general case, De Wannemacker 
\cite{wannemacker2} established the inequality
\begin{equation}
\nu_{2} \left( S(n,k) \right) \geq s_{2}(k)-s_{2}(n), \quad 0 \leq k \leq n. 
\label{wann}
\end{equation}
\noindent
Figure \ref{wann-101} shows the difference 
$\nu_{2}(S(n,k)) - s_{2}(k)+s_{2}(n)$ in the cases $k=101$ and  $k=129$.

{{
\begin{figure}[ht]
\subfigure{\includegraphics[width=3in]{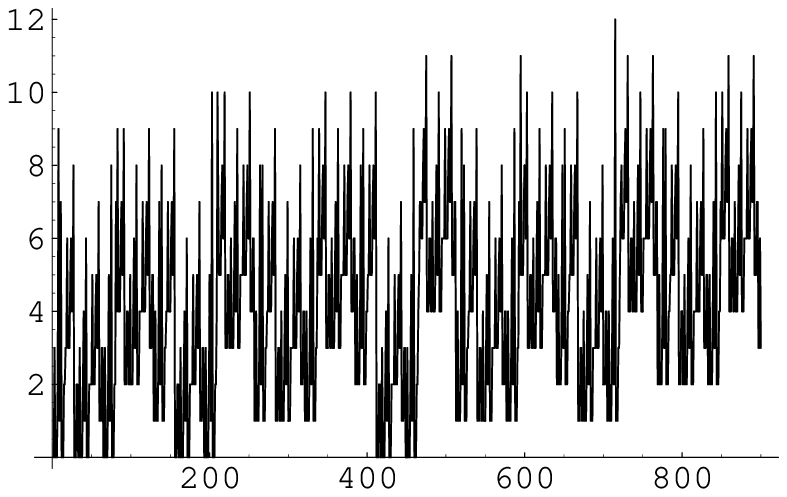}}
\subfigure{\includegraphics[width=3in]{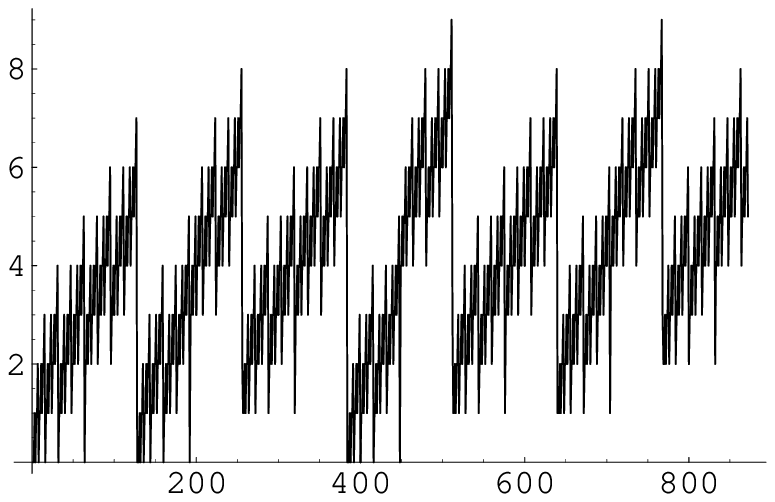}}
\caption{De Wannemacker difference for $k = 101$ and $k=129$}
\label{wann-101}
\end{figure}
}}

\section{The elementary cases} \label{sec-elem} 
\setcounter{equation}{0}

This section presents, for sake of completeness, the $2$-adic valuation 
of $S(n,k)$ for $1 \leq k \leq 4$. The formulas for $S(n,k)$ come 
from (\ref{formulastir}). The arguments are all elementary. 

\begin{lemma}
\label{stirvalue1}
The Stirling numbers of order $1$ are given by $S(n,1) = 1$, for all 
$n \in \mathbb{N}$. Therefore 
\begin{equation}
\nu_{2}(S(n,1)) = 0. 
\end{equation}
\end{lemma}
\begin{lemma}
\label{stirvalue2}
The Stirling numbers of order $2$ are given by $S(n,2) = 2^{n}-1$, for all 
$n \in \mathbb{N}$. Therefore 
\begin{equation}
\nu_{2}(S(n,2)) = 0. 
\end{equation}
\end{lemma}
\begin{lemma}
\label{stirvalue3}
The Stirling numbers of order $3$ are given by 
\begin{equation}
S(n,3) = \tfrac{1}{2} ( 3^{n-1} - 2^{n} + 1 ). 
\end{equation}
\noindent
Moreover,
\begin{equation}
\nu_{2}(S(n,3)) =  \begin{cases}
                0 \quad \text{ if } \, $n$ \quad \text{ is odd},\\
                1 \quad \text{ if }  \, $n$ \quad \text{ is even}. 
             \end{cases}
\end{equation}
\end{lemma}
\begin{proof}
Iterate the recurrence (\ref{recurrence}) to obtain
\begin{equation}
2^{n}-1 = S(n,3) - \sum_{k=1}^{N-1} 3^{k} (2^{n-k}-1) - 3^{N}S(n-N,3),
\end{equation}
\noindent
and with $N=n-1$ we have
\begin{equation}
S(n,3) = 2^{n}-1 - \sum_{k=1}^{n-2} 3^{k} (2^{n-k}-1). 
\end{equation}
\noindent
If $n$ is odd, then $S(n,3)$ is odd and $\nu_{2}(S(n,3)) = 0$. 

For $n$ even, the recurrence (\ref{recurrence}) yields
\begin{equation}
S(n,3) = 2^{n-1} + 3 \cdot 2^{n-2} -4 + 3^{2}S(n-2,3). 
\label{odd-ind}
\end{equation}
\noindent
As an inductive step, assume that $S(n-2,3) = 2T_{n-2}$, with $T_{n-2}$ odd.
Then (\ref{odd-ind}) yields 
\begin{equation}
\tfrac{1}{2} S(n,3) = 2^{n-2} + 3 \cdot 2^{n-3} + 3^{2}T_{n-2} -2,
\end{equation}
\noindent
and we conclude that $S(n,3)/2$ is an odd integer. Therefore 
$\nu_{2}(S(n,3)) = 1$ as claimed.  
\end{proof}

We now present a second proof of this result using elementary properties of 
the valuation $\nu_{2}$. In particular, we use the ultrametric inequality
\begin{equation}
\nu_{2}(x_{1}+x_{2}) \geq \text{Min} \left\{ \nu_{2}(x_{1}), \, \nu_{2}(x_{2}) 
\, \right\}. 
\label{ultramet}
\end{equation}
\noindent
The inequality is strict unless $\nu(x_{1}) = \nu_{2}(x_{2})$. This inequality
is equivalent to (\ref{metric1}). 

\vskip 0.1in

\noindent
The powers of $3$ modulo $8$ satisfy
$3^{m} + 1 \equiv 2 + (-1)^{m+1} \bmod 8$,
because $3^{2k} \equiv 1 \bmod 8$. Therefore, $3^{m}+1 = 8t + 3 +(-1)^{m+1}$
for some $t \in \mathbb{Z}$. Now 
\begin{equation}
\nu_{2}(8t) = 3 + \nu_{2}(t) > \nu_{2}( 3 + (-1)^{m+1} ), 
\end{equation}
\noindent
and the ultrametric inequality  (\ref{ultramet}) yields
\begin{equation}
\nu_{2}(3^{m}+1) = \nu_{2}(3 + (-1)^{m+1}) = 
\begin{cases}
2 \quad \text{ if } m \text{ is odd}, \\ 
1 \quad \text{ if } m \text{ is even}. 
\end{cases}
\label{nu2three}
\end{equation}
\noindent
Using $2S(n,3) = 3^{n-1}+1-2^{n}$ and
$\nu_{2}(2^{n}) = n > 2 \geq \nu_{2}(3^{n-1}+1)$, we conclude that
\begin{equation}
\nu_{2}(S(n,3)) = \nu_{2}(3^{n-1}+1 - 2^{n}) - 1 = \nu_{2}(3^{n-1}+1) - 1.
\end{equation}
\noindent
Lemma \ref{stirvalue3} now follows from (\ref{nu2three}).  \\

We now discuss the Stirling number of order $4$. 

\begin{lemma}
\label{stirvalue4}
The Stirling numbers of order $4$ are given by 
\begin{equation}
S(n,4) = \tfrac{1}{6} ( 4^{n-1} - 3^{n} - 3 \cdot 2^{n+1} - 1 ). 
\end{equation}
\noindent
Moreover,
\begin{equation}
\nu_{2}(S(n,4)) =  \begin{cases}
                1 \quad \text{ if } \, $n$ \quad \text{ is odd},\\
                0 \quad \text{ if }  \, $n$ \quad \text{ is even}. 
             \end{cases}
\end{equation}
\end{lemma}
\begin{proof}
The expression for $S(n,4)$ comes from (\ref{formulastir}). To establish the 
formula for $\nu_{2}(S(n,4))$, we use the recurrence (\ref{recurrence}) in the
case $k=4$:
\begin{equation}
S(n,4) = S(n-1,3) + 4S(n-1,4). \label{recurr-4}
\end{equation}
\noindent
For $n$ even, the value $S(n-1,3)$ is odd, so that $S(n,4)$ is odd and 
$\nu_{2}(S(n,4)) =0$. For $n$ odd, $S(n,4)$ is even, since $S(n-1,3)$ is even.
Then (\ref{recurr-4}), written as
\begin{equation}
\tfrac{1}{2}S(n,4) = \tfrac{1}{2}S(n-1,3) + 2S(n-1,4),
\label{two19}
\end{equation}
\noindent
and the value $\nu_{2}(S(n-1,3)) = 1$, show that the right hand side of 
(\ref{two19}) is odd,
yielding $\nu_{2}(S(n,4)) = 1$.
\end{proof}

\section{The Stirling numbers of order $5$} \label{sec-case5} 
\setcounter{equation}{0}

The elementary cases discussed in the previous section are the only ones
for which the $2$-adic valuation $\nu_{2}(S(n,k))$ is easy to compute. The 
graph in figure \ref{stirling-5} shows $\nu_{2}(S(n,5))$.

{{
\begin{figure}[ht]
\begin{center}
\includegraphics[width=3in]{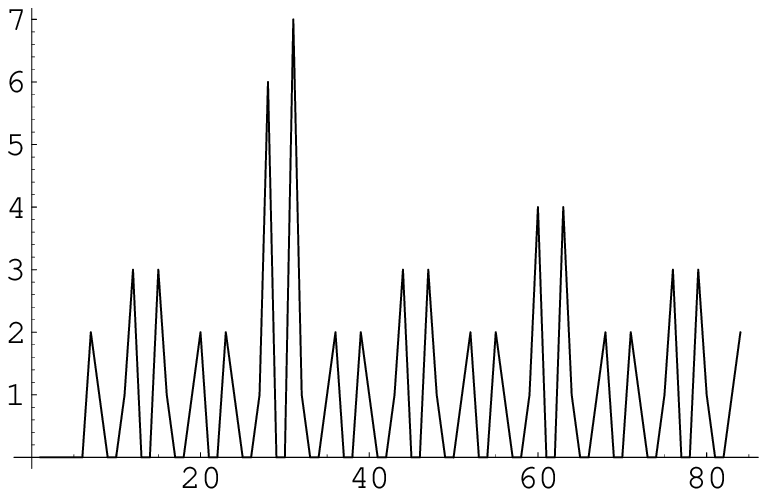}
\end{center}
\caption{The $2$-adic valuation of $S(n,5)$}
\label{stirling-5}
\end{figure}
}}

The explicit formula (\ref{formulastir}) yields
\begin{equation}
S(n,5) = \tfrac{1}{24}(5^{n-1}-4^{n} + 2 \cdot 3^{n} - 2^{n+1} + 1).
\end{equation}

\vskip 0.1in

We now discuss the valuation $\nu_{2}(S(n,5))$ in terms of the $m$-levels 
introduced in Section \ref{intro}. The $1$-level consists of 
two classes: $\{ C_{1,0}, \, C_{1,1} \}$. None of these classes are 
constant, so 
we split them into $\{ C_{2,0}, \, C_{2,1}, \, C_{2,2}, \, C_{2,3} \, \}$. The
parity of $S(n,5)$ determines two of them. 

\begin{lemma}
The Stirling numbers $S(n,5)$ satisfy
\begin{equation}
S(n,5) \equiv \begin{cases}
  1 \quad \bmod 2 \quad \text{ if } n \equiv 1, \text{ or } 2 \, \bmod 4, \\
  0 \quad \bmod 2 \quad \text{ if } n \equiv 3, \text{ or } 0 \, \bmod 4,
     \end{cases}
\end{equation}
\noindent
for all $n \in \mathbb{N}$.
\end{lemma}
\begin{proof}
The recurrence $S(n,5) = S(n-1,4) + 5S(n-1,5)$ and the parity
\begin{equation}
S(n,4) \equiv \begin{cases}
  1 \quad \bmod 2 \quad \text{ if } n \equiv 0 \, \bmod 2, \\
  0 \quad \bmod 2 \quad \text{ if } n \equiv 1 \, \bmod 2, \\
     \end{cases}
\end{equation}
\noindent
give the result by induction.
\end{proof}

\begin{corollary}
\label{4-vanish}
The $2$-adic valuations of the Stirling numbers $S(n,5)$ satisfy
\begin{equation}
\nu_{2}(S(4n+1,5)) = \nu_{2}(S(4n+2,5)) = 0 \text{ for all } n \in \mathbb{N}.
\end{equation}
\end{corollary}

The corollary states that the classes $C_{2,1}$ and $C_{2,2}$ are constant, so
the $2-$level is 
\begin{equation}
2-\text{level}: \quad \{ C_{2,0}, \, C_{2,3} \}. 
\end{equation}

This confirms part of the main conjecture: here $m_{0}=3$ in view of 
$2^{2} < 5 \leq 2^{3}$ and the first level where we find constant classes is 
$m_{0}-1=2$.  

\vskip 0.1in

\noindent
{\bf Remark}. Corollary \ref{4-vanish} reduces the discussion of 
$\nu_{2}(S(n,5))$ to the indices $n \equiv 0 \text{ or } 3 \bmod 4$. These 
two branches can be treated in parallel. Introduce the notation
\begin{equation}
q_{n} := \nu_{2}(S(n,5)), 
\end{equation}
\noindent
and consider the table of values 
\begin{equation}
X := \{ q_{4i}, \, q_{4i+3}: \, i \geq 2 \}.
\end{equation}
\noindent
This starts as
\begin{equation}
X = \{ 1, \, 1, \, 3, \, 3, \, 1, \, 1, \, 2, \, 2, \, 1, \, 1, \, 6, \, 
{\mathbf{7}}, \, 1, \, 1, \, \ldots \},
\end{equation}
\noindent
and after a while it continues as 
\begin{equation}
X = \{ \ldots, 1, \, 1, \, 2, \, 2, \, 1, \, 1, \, {\mathbf{11}}, \, 6, 
\, 1, \, 1, \, 2, \, 2, \, \ldots \}.
\end{equation}
\noindent
We observe that $q_{4i} = q_{4i+3}$ for most indices. 

\begin{definition}
The index $i$ is called {\em exceptional} if $q_{4i} \neq q_{4i+3}$.  
\end{definition}

The first exceptional index is $i=7$ where $q_{28} = 6 \neq q_{31} = 7$. The 
list of exceptional indices continues as 
$\{ 7, \, 39, \, 71, \, 103, \, \ldots \}$.

\begin{conjecture}
The set of exceptional indices is $\{ 32j + 7: \, j \geq 1 \}$.  
\end{conjecture}

We now consider the class 
\begin{equation}
C_{2,0} := \{ q_{4i} = \nu_{2}(S(4i),5): \, i \geq 2 \},
\end{equation}
\noindent
where we have omitted the first term $S(4,5) = 0$. The class $C_{2,0}$
starts as
\begin{equation}
C_{2,0} = \{ 1, \, 3, \, 1, \, 2, \, 1, \, 6, \, 1, \, 2, \, 1, \, 3, \, 
1, \, 2, \, 1, \, 4, \, 1, \, 2, \, 1, \, 3, \, 1, \, 2, \ldots \},
\end{equation}
\noindent
and it splits
according to the parity of the index $i$ into 
\begin{equation}
C_{3,4} = \{ q_{8i+4}: \, i \geq 1 \} \text{ and } 
C_{3,0} = \{ q_{8i}: \, i \geq 1 \}. 
\end{equation}
\noindent
It is easy to check that $C_{3,0}$ is constant.

\begin{proposition}
\label{prop1}
The Stirling numbers of order $5$ satisfy
\begin{equation}
\nu_{2}(S(8i,5)) = 1 \text{ for all } i \geq 1. 
\end{equation}
\end{proposition}
\begin{proof}
We analyze the identity 
\begin{equation}
24S(8i,5) = 5^{8i-1} - 4^{8i} + 2 \cdot 3^{8i} - 2^{8i+1} + 1
\end{equation}
\noindent
modulo $32$. Using $5^{8} \equiv 1$ and $5^{7} \equiv 13$, we obtain 
$5^{8i-1} \equiv 13$. Also, $4^{8i} \equiv 2^{8i+1} 
\bmod 0$. Finally,
$3^{8i} \equiv 81^{2i} \equiv 17^{2i} \equiv 1$. Therefore 
\begin{equation}
5^{8i-1} - 4^{8i} + 2 \cdot 3^{8i} - 2^{8i+1} + 1 \equiv 16 \bmod 32.
\end{equation}
\noindent
We obtain that $24S(8i,5) = 32t+16$ for some $t \in \mathbb{N},$ and this 
yields $3S(8i,5) = 2(2t+1)$. Therefore $\nu_{2}(S(8i,5)) = 1$. 
\end{proof}

We now consider the class $C_{3,4}$.

\begin{proposition}
\label{prop1a}
The Stirling numbers of order $5$ satisfy
\begin{equation}
\nu_{2}(S(8i+4,5)) \geq 2 \text{ for all } i \geq 1. 
\end{equation}
\end{proposition}
\begin{proof}
We analyze the identity 
\begin{equation}
24S(8i+4,5) = 5^{8i+3} - 4^{8i+4} + 2 \cdot 3^{8i+4} - 2^{8i+5} + 1
\end{equation}
\noindent
modulo $32$. Using $5^{8} \equiv 1, \, 5^{3} \equiv 29, \, 3^{8} \equiv 1, 
\, 3^{4} \equiv 17$ and $2^{4} \equiv 16$ modulo $32$, we obtain
\begin{equation}
24 S(8i+4,5) \equiv 0 \bmod 32.
\end{equation}
\noindent
Therefore $24S(8i+4,5) = 32t$ for some $t \in \mathbb{N}$, and this 
yields  $\nu_{2}(S(8i+4,5) \geq 2$. 
\end{proof}

\noindent
{\bf Note}. Lengyel \cite{lengyel1} established that
\begin{equation}
\nu_{2}(k! S(n,k)) = k-1,
\end{equation}
\noindent
for $n= a2^{q}$, $a$ odd and $q \geq k-2$. In the special case $k=5$, this
yields $\nu_{2}(S(n,5)) =1$ for $n = a2^{q}$ and $q \geq 3$. These values of 
$n$ have the form $n = 8a \cdot 2^{q-3}$, so this is included in 
Proposition \ref{prop1}. 

\vskip 0.1in

\noindent
{\bf Remark}. A similar argument yields 
\begin{equation}
\nu_{2}(S(8i+3,5)) = 1 \text{ and } \nu_{2}(S(8i+7,5)) \geq 2. 
\end{equation}
\noindent
We conclude that 
\begin{equation}
3-\text{level}: \quad \{ C_{3,4}, \, C_{3,7} \}. 
\end{equation}
\noindent
This illustrates the main conjecture: each of the classes of the $2$-level 
produces a constant class and a second one in the $3$-level.

We now consider the class $C_{3,4}$ and its splitting as $C_{4,4}$ and 
$C_{4,12}$. The data for $C_{3,4}$ starts as 
\begin{equation}
C_{3,4} = \{ 3, \, 2, \, 6, \, 2, \, 3, 
\, 2, \, 4, \, 2, \, 3, \, 2, \, 5, \, 2, \, 3, \, 2, \, 4, \, 2, \, 3, \, 2, 
\, 11, \, 2, \, 3, \, 2, \ldots \}. 
\end{equation}
\noindent
This suggests that the values with even index are all $2$. This is
verified below.   

\begin{proposition}
\label{prop2}
The Stirling numbers of order $5$ satisfy
\begin{equation}
\nu_{2}(S(16i+4,5)) = 2 \text{ for all } i \geq 1. 
\end{equation}
\end{proposition}
\begin{proof}
We analyze the identity 
\begin{equation}
24S(16i+4,5) = 5^{16i+3} - 4^{16i+4} + 2 \cdot 3^{16i+4} - 2^{16i+5} + 1
\end{equation}
\noindent
modulo $64$. Using $5^{16} \equiv 1, \, 5^{3} \equiv 61, \, 3^{16} \equiv 1$ 
and $3^{4} \equiv 17$, we obtain
\begin{equation}
5^{16i+3} - 4^{16i+4} + 2 \cdot 3^{16i+4} - 2^{16i+5} + 1 \equiv 32
\bmod 64. 
\end{equation}
\noindent
Therefore $24S(16i+4,5) = 64t+32$ for some $t \in \mathbb{N}$. This gives
$3S(16i+4,5) = 4(2t+1)$, and it follows that $\nu_{2}(S(16i+4,5)) = 2$. 
\end{proof}

\noindent
{\bf Note}. A similar argument shows that $\nu_{2}(S(16i+12,5)) \geq 3,  \,
\nu_{2}(S(16i+7,5)) = 2$ and $\nu_{2}(S(16i+15,5)) \geq 3$. Therefore
the $4$-level is $\{ C_{4,12}, \, C_{4,15} \}$. \\

This splitting process of the classes can be continued and, according to our 
main conjecture, the number of elements in the $m$-level is always 
constant. To prove the 
statement similar to Propositions \ref{prop1} and \ref{prop2}, we must
analyze the congruence 
\begin{equation}
24S(2^{m}i+j,5) \equiv  5^{2^{m}i+j-1} - 4^{2^{m}i+j} 
+ 2 \cdot 3^{2^{m}i+j} - 
2^{2^{m}i+j+1} +1 \text{ mod } 2^{m+2}. 
\label{messcong}
\end{equation}
\noindent
We present a proof of this conjecture, for the special case 
$k=5$, in the next section. 

\vskip 0.1in

Lundell \cite{lundell1} studied the Stirling-like numbers
\begin{equation}
T_{p}(n,k) = \sum_{j=0}^{k} 
(-1)^{k-j} \binom{k}{j} j^{n},
\end{equation}
\noindent
where the prime $p$ is fixed, and the index $j$ is omitted in the sum 
if it is divisible by $p$. 
Clarke \cite{clarke1} conjectured that 
\begin{equation}
\nu_{p}(k! \, S(n,k)) = \nu_{p}(T(n,k)). 
\label{clarke-conj}
\end{equation}
\noindent
From this conjecture he derives an expression for $\nu_{2}(S(n,5))$ in 
terms of the zeros of the form $f_{0,5}(x) = 5 + 10 \cdot 3^{x} + 5^{x}$ 
in the ring of $2$-adic integers $\mathbb{Z}_{2}$. 

\begin{theorem}
Let $u_{0}$ and $u_{1}$ be the $2$-adic zeros of the function $f_{0,5}$. 
Then, under the assumption that conjecture (\ref{clarke-conj}) holds, we 
have
\begin{equation}
\nu_{2}(S(n,5)) = \begin{cases} 
         -1 + \nu_{2}(n-u_{0}) \quad \text{ if } n \text{ is even}, \\
         -1 + \nu_{2}(n-u_{1}) \quad \text{ if } n \text{ is odd}. 
      \end{cases}
\end{equation}
\noindent
Here $u_{0}$ is the unique zero of $f_{0,5}$ that satsifies $u_{0} \in 
2 \mathbb{Z}_{2}$ and $u_{1}$ is the other zero of $f_{0,5}$ and satisfies 
$u_{1} \in 1 + 2 \mathbb{Z}_{2}$. 
\end{theorem}

Clarke also obtained in \cite{clarke1} similar 
expressions for $\nu_{2}(S(n,6))$ and $\nu_{2}(S(n,7))$
in terms of zeros of the functions 
\begin{equation}
f_{0,6} = -6-20 \cdot 3^{x} - 6 \cdot 5^{x} \text{ and }
f_{0,7} = 7 + 35 \cdot 3^{x} + 21 \cdot 5^{x} + 7^{x}. \nonumber
\end{equation}

\section{Proof of the main conjecture for $k=5$} \label{sec-conjecture} 
\setcounter{equation}{0}

The goal of this section is to prove the main conjecture in the case $k=5$.
The parameter $m_{0}$ is $3$ in view of 
$2^{2} < 5 \leq 2^{3}$. In the previous 
section we have verified that $m_{0}-1=2$ is the first level for constant 
classes. We now prove this  splitting of classes.

\begin{theorem}
\label{conj-5}
Assume $m \geq m_{0}$. Then the $m$-level consists of 
exactly two split classes: $C_{m,j}$ and 
$C_{m,j+2^{m-1}}$. They satisfy $\nu_{2}(C_{m,j}) > m-3$ and 
$\nu_{2}(C_{m,j+2^{m-1}}) > m-3$. Then exactly one, call it $C^{1}$, satisfies 
$\nu_{2}(C^{1}) = \{ m-2 \}$ and the other one, call it $C^{2}$, 
satisfies $\nu_{2}(C^{2}) > m-2$. 
\end{theorem}

The proof of this theorem requires several elementary results of $2$-adic 
valuations.  

\begin{lemma}
\label{lemma-a1}
For $m \in \mathbb{N}, \, \nu_{2} \left( 5^{2^{m}} -1 \right) = m+2$. 
\end{lemma}
\begin{proof}
Start at $m=1$ with $\nu_{2}(24) = 3$. The inductive step uses 
\begin{equation}
5^{2^{m+1}}-1 = ( 5^{2^{m}}-1 ) \cdot ( 5^{2^{m}} +1 ). 
\nonumber
\end{equation}
\noindent
Now $5^{k} + 1 \equiv 2 \bmod 4$ so that $5^{2^{m}}+1 = 2 \alpha_{1}$ with 
$\alpha_{1}$ odd. Thus
\begin{equation}
\nu_{2} ( 5^{2^{m+1}}-1 )  = 
\nu_{2}( 5^{2^{m}}-1 )  + \nu_{2}( 5^{2^{m}} +1 ) 
= (m+2) +1 = m+3.
\nonumber 
\end{equation}
\end{proof}

The same type of argument produces the next lemma. 

\begin{lemma}
\label{lemma-a2}
For $m \in \mathbb{N}, \,  \nu_{2} ( 3^{2^{m}} -1 ) = m+2$. 
\end{lemma}

\begin{lemma}
\label{lemma-a3}
For $m \in \mathbb{N}, \, \nu_{2} ( 5^{2^{m}} - 3^{2^{m}} ) = m+3$. 
\end{lemma}
\begin{proof}
The inductive step uses 
\begin{equation}
5^{2^{m+1}} - 3^{2^{m+1}} = ( 5^{2^{m}} - 3^{2^{m}} ) \times 
\left( ( 5^{2^{m}}-1 ) + ( 3^{2^{m}} + 1 ) \right). 
\nonumber
\end{equation}
\noindent
Therefore $\nu_{2}( 5^{2^{m}}-1) = m+2$ and $3^{2^{m}} 
\equiv 1 \bmod 4$, thus $\nu_{2}(3^{2^{m}} + 1) = 1$. We conclude that
\begin{equation}
\nu_{2} ( ( 5^{2^{m}} -1) + (3^{2^{m}}+1) ) = 
\text{Min} \{ m+2, 1 \} = 1. \nonumber
\end{equation}
\noindent
We obtain
\begin{equation}
\nu_{2} ( 5^{2^{m+1}} - 3^{2^{m+1}} ) = m+4, 
\end{equation}
\noindent
and this concludes the inductive step.
\end{proof}

The recurrence (\ref{recurrence}) for the Stirling numbers $S(n,5)$  is
$S(n,5) = 5S(n-1,5) + S(n-1,4)$. Iterating yields the next lemma.

\begin{lemma}
\label{recurr-5}
Let $t \in \mathbb{N}$. Then
\begin{equation}
S(n,5) - 5^{t} S(n-t,5) = \sum_{j=0}^{t-1} 5^{j} S(n-j-1,4). 
\end{equation}
\end{lemma}

\noindent
{\bf Proof of theorem \ref{conj-5}}. We have already checked the conjecture for
the $2$-level. The inductive hypothesis states that there is an 
$(m-1)$-level survivor of the form
\begin{equation}
C_{m,k} = \{ \nu_{2} ( S(2^{m}n+k,5) ): \, n \geq 1 \},
\end{equation}
\noindent
where $\nu_{2} ( S(2^{m}n+k,5) ) > m-2$. At the next level, $C_{m,k}$
splits into the two classes 
\begin{eqnarray}
C_{m+1,k} & = & \{ \nu_{2} ( S(2^{m+1}n+k,5) ): \, n \geq 1 \} \quad 
\text{ and }
\nonumber \\
C_{m+1,k+2^{m}} & = & \{ \nu_{2} ( S(2^{m+1}n+k+2^{m},5) ): \, n \geq 1 \},
\nonumber 
\end{eqnarray}
\noindent
and every element of each of these two classes is greater or equal 
to $m-1$. We now prove that one of these classes reduces to 
the singleton $\{ m-1 \}$ 
and that every element in the other class is strictly greater than $m-1$.

The first step is to use Lemma \ref{recurr-5} to compare the values of 
$S(2^{m+1}n+k,5)$ and $S(2^{m+1}n+k+2^{m},5)$. Define 
\begin{equation}
M = 2^{m}-1 \text{ and } N = 2^{m+1}n+k,
\end{equation}
\noindent
and use (\ref{recurrence}) to write
\begin{equation}
S(2^{m+1}n+k+2^{m},5) - 5^{2^{m}} S(2^{m+1}n+k,5) = 
\sum_{j=0}^{M} 5^{M-j} S(N+j,4). 
\label{compare}
\end{equation}
\noindent
The next proposition establishes the $2$-adic valuation of the right hand side.

\begin{proposition}
\label{nu2sum}
With the notation as above,
\begin{equation}
\nu_{2} \left( \sum_{j=0}^{M} 5^{M-j} S(N+j,4) \right) = m-1.
\end{equation}
\end{proposition} 
\begin{proof}
The explicit formula (\ref{formulastir}) yields 
$6 S(n,4) = 4^{n-1} + 3 \cdot 2^{n-1} - 3^{n} -1$. Thus 
\begin{eqnarray}
6 \sum_{j=0}^{M} 5^{M-j} S(N+j,4) & = & 4^{N-1} ( 5^{M+1}-4^{M+1}) + 
2^{N-1}(5^{M+1}-2^{M+1}) \nonumber \\
& & - 3^{N} \times \tfrac{1}{2} (5^{M+1}-3^{M+1}) - 
\tfrac{1}{4}(5^{M+1}-1). 
\nonumber
\end{eqnarray}
\noindent
The results in Lemmas \ref{lemma-a1}, \ref{lemma-a2} and \ref{lemma-a3} 
yield 
\begin{equation}
6 \sum_{j=0}^{M} 5^{M-j} S(N+j,4) = 
4^{N-1} \alpha_{1} + 2^{N-1}\alpha_{2} - 3^{N} \cdot 2^{m+2} \alpha_{3} - 
2^{m} \alpha_{4},
\end{equation}
\noindent
with $\alpha_{j}$ odd integers. Write this as
\begin{equation}
6 \sum_{j=0}^{M} 5^{M-j} S(N+j,4) = 
2^{N-1} \left( 2^{N-1}\alpha_{1} + \alpha_{2} \right) - 2^{m} 
\left( 4 \alpha_{3} 3^{N} + 1 \right) \equiv T_{1} + T_{2}. \nonumber
\end{equation}
\noindent
Then $\nu_{2}(T_{1}) = N-1 > m = \nu_{2}(T_{2}),$ and we obtain
\begin{equation}
\nu_{2} \left( \sum_{j=0}^{M} 5^{M-j} S(N+j,4) \right) = m-1. 
\end{equation}
\noindent
We conclude that
\begin{equation}
S( 2^{m+1}n+k+2^{m},5) - 5^{2^{m}} S(2^{m+1}n+k,5) = 2^{m-1} \alpha_{5},
\end{equation}
\noindent
with $\alpha_{5}$ odd. Define 
\begin{equation}
X := 2^{-m+1}S(2^{m+1}n+k+2^{m},5) \text{ and }
Y := 2^{-m+1}S(2^{m+1}n+k,5). 
\end{equation}
\noindent
Then $X$ and $Y$ are integers and $X - Y \equiv 1 \bmod 2$, so that  they 
have opposite parity. If $X$ is even and $Y$ is odd, we obtain
\begin{equation}
\nu_{2} \left( S(2^{m+1}n+k+2^{m},5) \right) > m-1 \text{ and } 
\nu_{2} \left( S(2^{m+1}n+k,5) \right) =  m-1. 
\end{equation}
\noindent
The case $X$ odd and $Y$ even is similar. This completes the proof.
\end{proof}

\medskip

There are four classes at the first level corresponding to the residues 
modulo $4$, two of which are constant. The complete determination of the 
valuation $\nu_{2}(S(n,5))$ is now determined by the choice of class 
when we move from level $m$ to $m+1$. We consider only the branch starting 
at indices congruent to $0$ modulo $4$; the case of $3$ modulo $4$ is 
similar. Now there is 
single class per level that we write as
\begin{equation}
C_{m,j} = \{ q_{2^{m}i+j}: \, i \in \mathbb{N} \},
\label{surv-class1}
\end{equation}
\noindent
where $j= j(m)$ is the index that corresponds to the non-constant class at the
$m$-level. The first few examples are listed below.
\begin{eqnarray}
C_{2,4} & = & \{ q_{4i+4}: \, i \in \mathbb{N} \} \nonumber \\
C_{3,4} & = & \{ q_{8i+4}: \, i \in \mathbb{N} \} \nonumber \\
C_{4,12} & = & \{ q_{16i-4}: \, i \in \mathbb{N} \} \nonumber \\
C_{5,28} & = & \{ q_{32i-4}: \, i \in \mathbb{N} \} \nonumber \\
C_{6,28} & = & \{ q_{64i-36}: \, i \in \mathbb{N} \} \nonumber \\
C_{7,156} & = & \{ q_{128i-100}: \, i \in \mathbb{N} \} \nonumber \\
C_{8,156} & = & \{ q_{256i-100}: \, i \in \mathbb{N} \} \nonumber \\
C_{9,156} & = & \{ q_{512i-356}: \, i \in \mathbb{N} \} \nonumber \\
C_{10,156} & = & \{ q_{1024i-868}: \, i \in \mathbb{N} \} \nonumber 
\end{eqnarray}
\noindent

We have observed a connection between the indices $j(m)$ and the set of 
exceptional indices $I_{1}$ in (\ref{seti1}).  

\begin{conjecture}
Construct a list of numbers $\{ c_{i}: \, i \in \mathbb{N} \}$ according to 
the following rule: let $c_{1} = 8$ (the first index in the class $C_{2,4}$),
and then define $c_{j}$ as the first value on $C_{m,j}$ that is strictly 
bigger than $c_{j-1}$. The set $C$ begins as
\begin{equation}
C = \{ 8, \, 12, \, 28, \, 60, \, 92, \, 156, \, 412, \, 668, \, 1180, 
\ldots \}.
\end{equation}
\noindent
Then, starting at $156$, the number $c_{i} \in I_{1}$.
\end{conjecture}

\section{Some approximations} \label{sec-app} 
\setcounter{equation}{0}

In this section we present some approximations to the function 
$\nu_{2}(S(n,5))$. These approximations were derived empirically, and they
support our belief that $2$-adic valuations of Stirling numbers can be 
well approximated by simple integer combinations of $2$-adic 
valuations of integers.

For each prime $p$, define
\begin{equation}
\lambda_{p}(m) = \frac{1}{2} \left( 1 - (-1)^{m \text{ mod }p} \right).
\end{equation}

\noindent
{\bf First approximation}. Define 
\begin{equation}
f_{1}(m) := \lfloor{ \frac{m+1}{2} \rfloor} + 112 \lambda_{2}(m) + 
50 \lambda_{2}(m+1). 
\end{equation}
\noindent
Then $\nu_{2}(S(m,5))$ and $\nu_{2}(f_{1}(m))$ agree for most values. The first
time they differ is at $m = 156$ where 
\begin{equation}
\nu_{2}(S(156,5)) - \nu_{2}(f_{1}(156)) = 4. \nonumber
\end{equation}
\noindent
The first few indices for which $\nu_{2}(S(m,5)) \neq \nu_{2}(f_{1}(m))$ are
$\{ 156, \, 287, \, 412, \, 668, \, 799, \, \ldots \}$. 

\begin{conjecture}
\label{conj1-app}
Define 
\begin{equation}
x_{1}(m) = 156 + 125 \lfloor{ \frac{4m}{3} \rfloor} + 
6 \lfloor{ \frac{2m+1}{3} \rfloor}
\end{equation}
\noindent
and 
\begin{equation}
I_{1} = \{ x_{1}(m): \, m \geq 0 \}. 
\label{seti1}
\end{equation}
\noindent
Then $\nu_{2}(S(m,5)) = \nu_{2}(f_{1}(m))$ unless $m \in I_{1}$. 
\end{conjecture}

The parity of the exceptions in $I_{1}$ 
is easy to establish: every third element is odd and the even indices of 
$I_{1}$ are on the arithmetic progression $256m+156$. 

\vskip 0.1in

\noindent
{\bf Second approximation}. We now consider the error 
\begin{equation}
Err_{1}(m,5) := \nu_{2}(S(m,5)) - \nu_{2}(f_{1}(m)). 
\end{equation}
\noindent
Observe that $Err_{1}(m,5) = 0$ for $m$ outside $I_{1}$.

\vskip 0.1in

Define 
\begin{eqnarray}
m_{3}(m) & := & (m+2) \, \bmod 3, \nonumber \\
\alpha_{m} & := & \lambda_{3}(m+2) \left( 1 + \lambda_{3}(m) \right) + 
\lambda_{2}(m+1) \lambda_{3}(m), \nonumber 
\end{eqnarray}
\noindent
and
\begin{equation}
f_{2}(m) = \binom{2m_{3}}{m_{3}} \lfloor{ \frac{m+2}{3} \rfloor} + 
208 \lambda_{3}(m+1) + 27 \lambda_{2}(m)\lambda_{3}(m). 
\end{equation}

The next conjecture improves the prediction of Conjecture \ref{conj1-app}. 

\begin{conjecture}
Consider the set 
$I_{2} = \{ x_{2}(m): \, m \geq 0 \}$, where 
\begin{equation}
x_{2}(m) = 109 + 
107 \lfloor{\frac{4m+2}{3} \rfloor} + 85 \lfloor{ \frac{4m+1}{3} \rfloor}. 
\end{equation}
Then
\begin{equation}
\text{Err}_{1}(x_{1}(m),5) =  (-1)^{\alpha_{m}} \nu_{2}(f_{2}(m)),
\end{equation}
\noindent
unless $m \in I_{2}$. 
\end{conjecture}

We now present one final improvement. Define 
\begin{equation}
Err_{2}(m,5) := Err_{1}(x_{1}(m),5) - (-1)^{\alpha_{m}} 
\nu_{2}f_{2}(m)). 
\end{equation}
\noindent
Define $\beta(m) = \alpha_{m} + (-1)^{m+1} \lambda_{3}(m)$ and 
\begin{equation}
f_{3}(m) = 4^{1- \lambda_{3}(m)} \lfloor{ \frac{m+2}{3} \rfloor} + 
\lambda_{3}(m) \left( 85 \lambda_{3}(m) + 8 \lambda_{2}(m+1) + 
2 \lambda_{3}(m+1) \right). 
\end{equation}

\begin{conjecture}
$Err_{2}(x_{2}(m),5)$ agrees with $(-1)^{\beta(x_{2}(m))} \, 
\nu_{2}(f_{3}(x_{2}(m)))$ for most values of $m \in \mathbb{N}$. 
\end{conjecture}

\section{A sample of pictures} \label{sec-pictures} 
\setcounter{equation}{0}

In this section we present data that illustrate the wide variety of behavior 
for the $2$-adic valuation of Stirling numbers $S(n,k)$. Several features are 
common to all. For instance we observe the appearance of an {\em empty region}
from below the graph. The graph of $\nu_{2}(S(n,126))$ shows the basic 
function $\nu_{2}(n)$ in the interior of the graph. The dark objects in 
$\nu_{2}(S(n,195))$ and $\nu_{2}(S(n,260))$ correspond to an oscillation
between two consecutive values. We are lacking an explanation of 
these features.

{{
\begin{figure}[ht]
\subfigure{\includegraphics[width=3in]{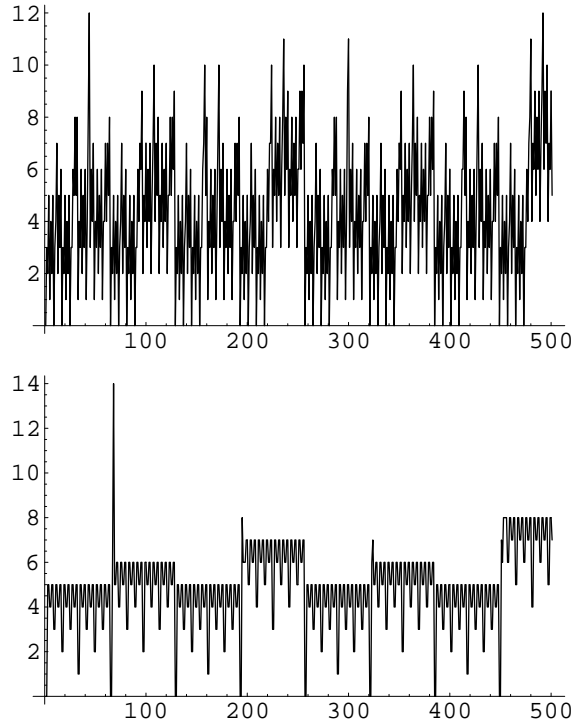}}
\subfigure{\includegraphics[width=3in]{500data126.eps}}
\caption{The data for $S(n,80)$ and $S(n,126)$.}
\label{sample80}
\end{figure}
}}

{{
\begin{figure}[ht]
\subfigure{\includegraphics[width=3in]{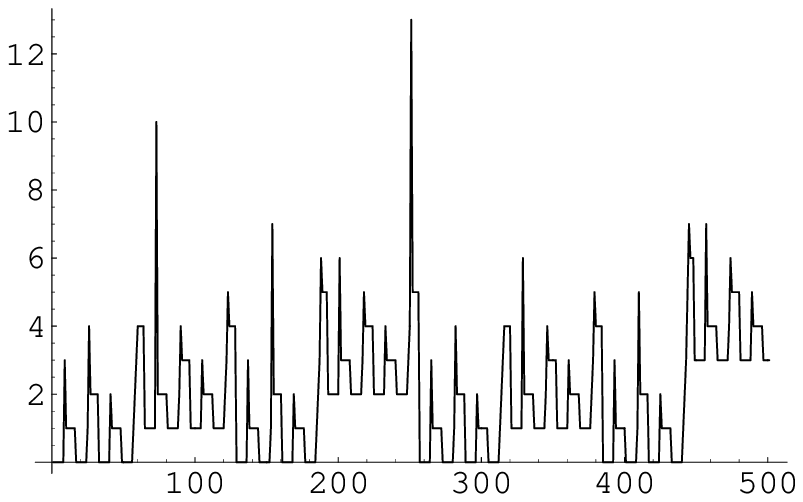}}
\subfigure{\includegraphics[width=3in]{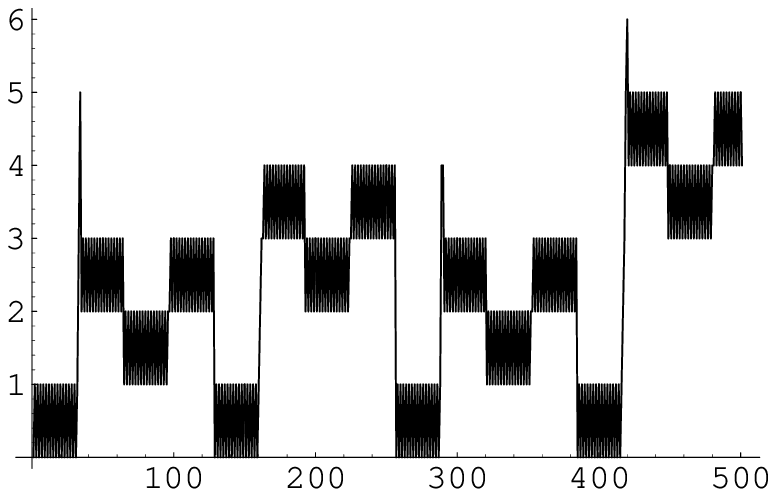}}
\caption{The data for $S(n,146)$ and $S(n,195)$.}
\label{sample146}
\end{figure}
}}

{{
\begin{figure}[ht]
\subfigure{\includegraphics[width=3in]{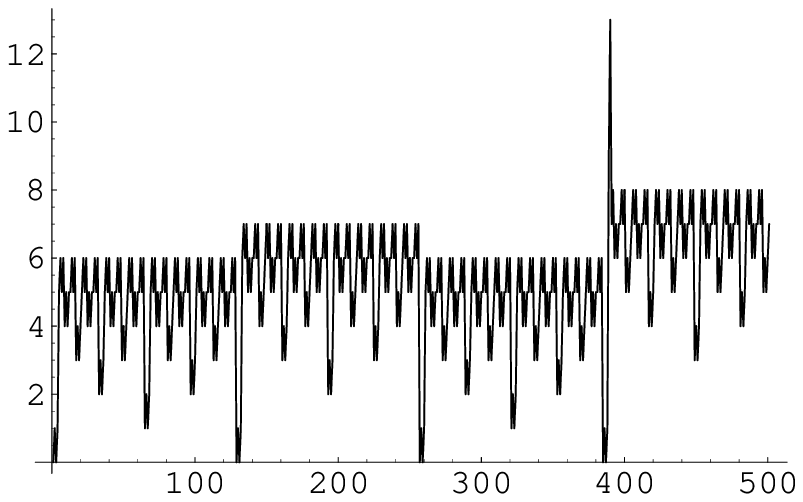}}
\subfigure{\includegraphics[width=3in]{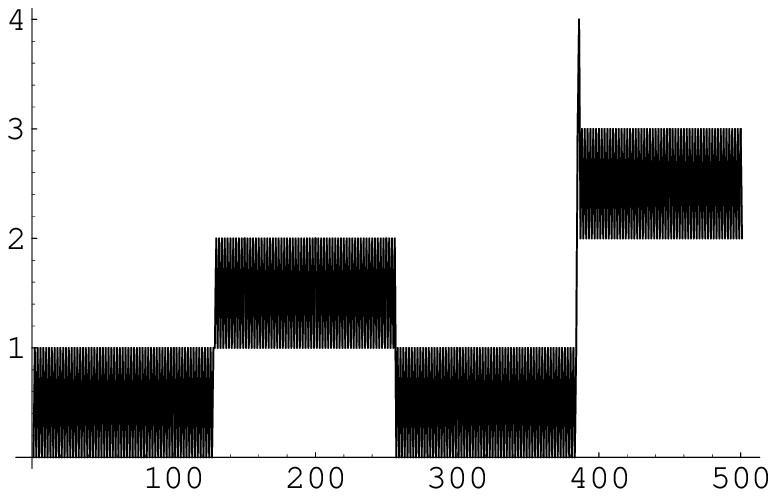}}
\caption{The data for $S(n,252)$ and $S(n,260)$.}
\label{sample252}
\end{figure}
}}

{{
\begin{figure}[ht]
\subfigure{\includegraphics[width=3in]{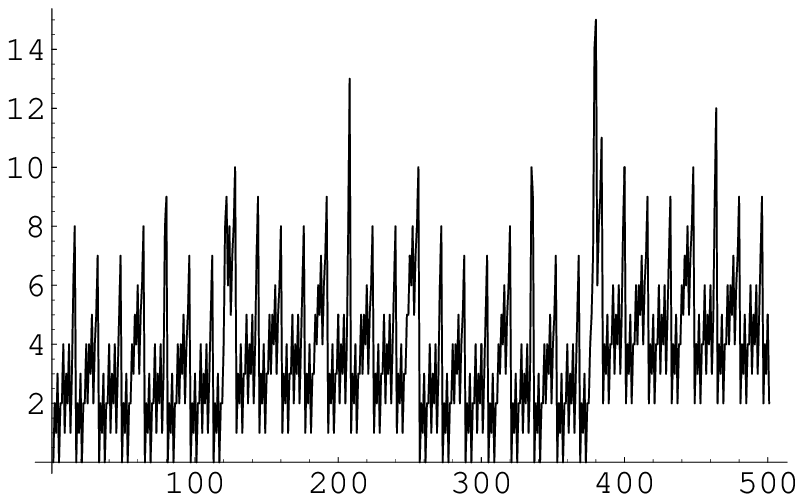}}
\subfigure{\includegraphics[width=3in]{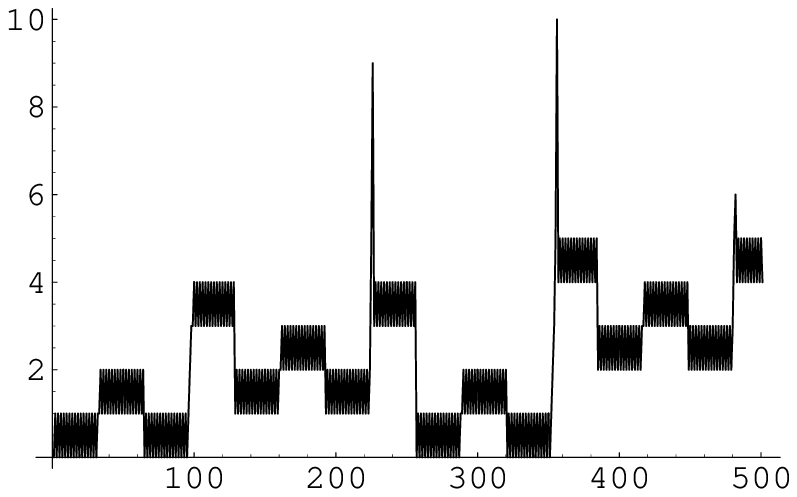}}
\caption{The data for $S(n,279)$ and $S(n,324)$.}
\label{sample279}
\end{figure}
}}

{{
\begin{figure}[ht]
\subfigure{\includegraphics[width=3in]{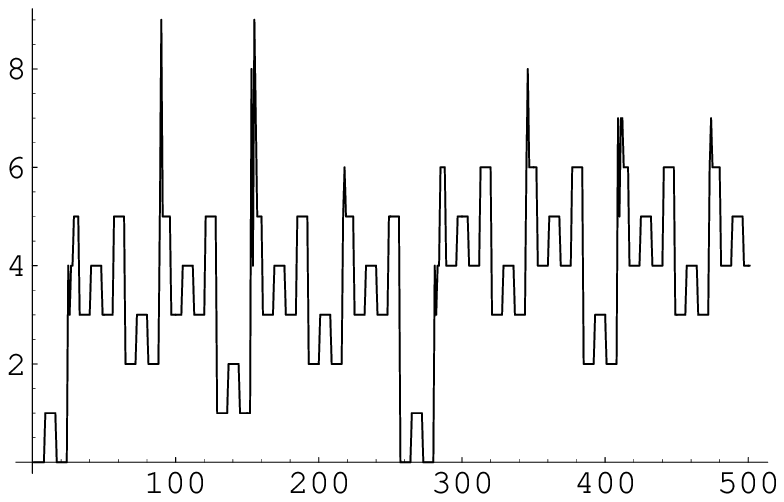}}
\subfigure{\includegraphics[width=3in]{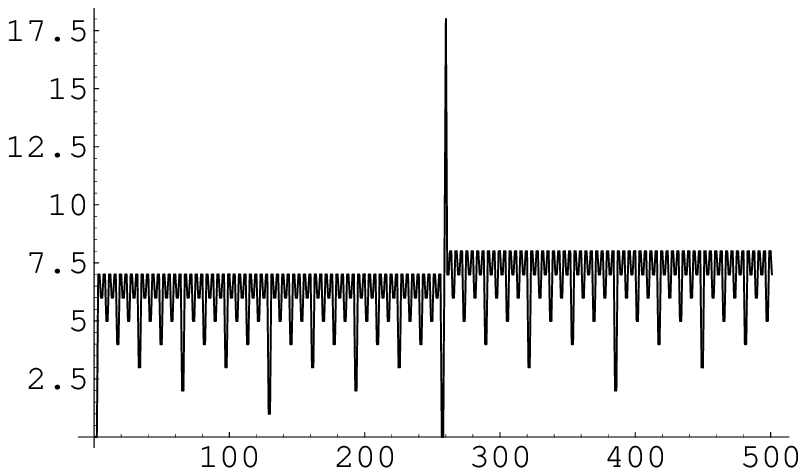}}
\caption{The data for $S(n,465)$ and $S(n,510)$.}
\label{sample465}
\end{figure}
}}

\section{Conclusions} \label{sec-conclusions} 
\setcounter{equation}{0}

We have presented a conjecture that describes the $2$-adic valuation of 
the Stirling numbers $S(n,k)$. This conjecture is established for $k=5$.  \\

{\bf Acknowledgements}. The last author acknowledges the partial support of 
NSF-DMS 0409968. The second author was partially supported as  a graduate 
student by the same grant. The work of the first author was done while 
visiting Tulane University in the Spring of 2006.  The authors wish to thank
Valerio de Angelis for the diagrams in the paper.

\end{document}